\documentclass[11pt, a4paper]{amsart}

\usepackage{amssymb,amsmath,amsthm}
\usepackage{tikz-cd}
\usetikzlibrary{arrows}
\usepackage{fullpage}
\usepackage{graphicx}
\usepackage[hidelinks]{hyperref}
\usepackage{mathtools}
\usepackage[titletoc,title]{appendix}
\usepackage{graphicx}
\usepackage{caption}
\usepackage{subcaption}
\usepackage{blindtext}
\usepackage[titletoc,title]{appendix}
\usepackage[english]{babel}
\usepackage{mathrsfs}
\usepackage{cite}
\usepackage{mathabx}
\usepackage{microtype}

\newtheorem{theorem}{Theorem}[section]
\newtheorem{lemma}[theorem]{Lemma}
\newtheorem{corollary}[theorem]{Corollary}
\newtheorem{proposition}[theorem]{Proposition}
\theoremstyle{definition}
\newtheorem{definition}[theorem]{Definition}
\newtheorem{example}[theorem]{Example}

\newtheorem{remark}[theorem]{Remark}

\newtheorem*{theorem*}{Theorem}
\newtheorem*{conjecture*}{Conjecture}

\newcommand{\isom}{\cong}

\newcommand{\N}{\mathbb{N}}
\newcommand{\Z}{\mathbb{Z}}
\newcommand{\Q}{\mathbb{Q}}
\newcommand{\R}{\mathbb{R}}
\makeatletter
\def\Ddots{\mathinner{\mkern1mu\raise\p@
\vbox{\kern7\p@\hbox{.}}\mkern2mu
\raise4\p@\hbox{.}\mkern2mu\raise7\p@\hbox{.}\mkern1mu}}

\makeatother

\newcommand{\normal}[1]{\langle\langle #1 \rangle\rangle}

\def\immerses{\looparrowright}
\def\injects{\hookrightarrow}

\DeclareSymbolFontAlphabet{\amsmathbb}{AMSb}

\DeclareMathOperator{\rk}{rk}

\DeclareMathOperator{\pr}{pr}

\DeclareMathOperator{\core}{Core}

\DeclareMathOperator{\bs}{BS}

\DeclarePairedDelimiter\abs{\lvert}{\rvert}

\makeatletter
\let\oldabs\abs
\def\abs{\@ifstar{\oldabs}{\oldabs*}}

\newcounter{cases}
\newcounter{subcases}[cases]
\newenvironment{mycase}
{
    \setcounter{cases}{0}
    \setcounter{subcases}{0}
    \newcommand{\case}
    {
        \par\indent\stepcounter{cases}\textbf{Case \thecases.}
    }
    
}
{
    \par
}
\renewcommand*\thecases{\arabic{cases}}

\newcommand{\fakeenv}{} 

\newenvironment{restate}[2]  
{ 
 \renewcommand{\fakeenv}{#2} 
 \theoremstyle{plain} 
 \newtheorem*{\fakeenv}{#1~\ref{#2}} 
 \begin{\fakeenv}
}
{
 \end{\fakeenv}
}

\begin{document}

\title{Hyperbolic one-relator groups}
\author{Marco Linton}

\date{\today}

\address{University of Oxford, Oxford, OX2 6GG, UK}

\email{marco.linton@maths.ox.ac.uk}

\begin{abstract}
We introduce two families of two-generator one-relator groups called primitive extension groups and show that a one-relator group is hyperbolic if its primitive extension subgroups are hyperbolic. This reduces the problem of characterising hyperbolic one-relator groups to characterising hyperbolic primitive extension groups. These new groups moreover admit explicit decompositions as graphs of free groups with adjoined roots. In order to obtain this result, we characterise $2$-free one-relator groups with exceptional intersection in terms of Christoffel words, show that hyperbolic one-relator groups have quasi-convex Magnus subgroups and build upon the one-relator tower machinery developed in the authors previous article.
\end{abstract}

\maketitle

\section{Introduction}

Since the introduction of hyperbolic groups by Gromov in the 80's, a wealth of powerful tools have been developed to study them. Thus, when studying a class of groups, a classification of those that are hyperbolic can be very useful. We here focus on the class of one-relator groups; that is, groups of the form $F(\Sigma)/\normal{w}$, where $F(\Sigma)$ denotes the free group generated by $\Sigma$. The best possible statement that one could hope for is known as Gersten's conjecture \cite{gersten_92} which asserts that one-relator groups without Baumslag--Solitar subgroups are hyperbolic. In this article, we present some progress in this direction, building on our previous work in \cite{me_22}.

A \emph{Magnus subgroup} of a one-relator group is a subgroup generated by a subset of the generators $\Sigma$, omitting at least one generator mentioned in the cyclic reduction of $w$. The theory of one-relator groups originated in 1930, when Magnus proved the Freiheitssatz \cite{magnus_30}. 

\begin{theorem*}[Freiheitssatz]
Magnus subgroups of one-relator groups are free.
\end{theorem*}

The proof of the Freiheitssatz makes use of a hierarchy of one-relator groups known as the Magnus hierarchy: a one-relator group $G$ splits as a HNN-extension with one-relator vertex group $H$ of lower `complexity' and where the edge groups are Magnus subgroups of $H$. By understanding the splittings that arise in this way, one may use induction to make conclusions about $G$. See \cite{magnus_04,lyn_00,me_22} for various versions and applications of the Magnus hierarchy. 

Almost a century later, the Magnus hierarchy is still a powerful tool for the study of one-relator groups. However, the splittings that arise remain somewhat mysterious. In \cite{me_22}, the case that $G$ is a $2$-free one-relator group was considered in detail: there it was shown that if $G$ is a finitely generated $2$-free one-relator group, then $G$ is hyperbolic and acts acylindrically on any Bass-Serre tree associated with a one-relator splitting. When $G$ has torsion, the same properties hold \cite{newman_68,wise_21,me_22}. On the other hand, the fact that Baumslag--Solitar groups cannot act acylindrically on a tree \cite{minasyan_15} implies that the general case is more complicated.

We say that a one-relator group $F(\Sigma)/\normal{w}$ has an \emph{exceptional intersection} if there are subsets $A, B\subset \Sigma$ that generate Magnus subgroups and such that the following holds in the quotient:
\[
\langle A\rangle\cap \langle B\rangle \neq \langle A\cap B\rangle.
\]
These groups were first studied by Collins in \cite{collins_04}. Howie then obtained generalisations in the context of one-relator products of locally indicable groups in \cite{howie_05}. Examples of one-relator groups with exceptional intersections include torus knot groups $\langle a, b\mid a^p = b^q\rangle$ and closed orientable surface groups of genus at least two $\langle a_1, b_1, ..., a_g, b_g \mid [a_1, b_1] = [a_2, b_2]...[a_g, b_g]\rangle$.

In \cite{me_22}, it is shown that if $G = H*_{\psi}$ is a one-relator splitting where $G$ does not contain any Baumslag--Solitar subgroups, $H$ is hyperbolic and the edge groups of the splitting are quasi-convex and do not have exceptional intersection, then $G$ acts acylindrically on its Bass-Serre tree and so is hyperbolic by \cite{bestvina_92_combination}. In light of this (and Theorem \ref{quasiconvex_magnus}), when attempting to understand the hyperbolicity of one-relator groups, the case of interest lies when $H$ has an exceptional intersection. In this article, we take a step in this direction and characterise $2$-free one-relator groups with exceptional intersection.

\begin{restate}{Theorem}{exceptional_characterisation}
Let $G$ be a one-relator group with exceptional intersection. One of the following holds:
\begin{enumerate}
\item\label{itm:not_neg} $G$ is $2$-free,
\item\label{itm:has_GBS} there is a two-generator one-relator generalised Baumslag--Solitar subgroup $H<G$ such that every non-free two-generator subgroup of $G$ is conjugate into $H$.
\end{enumerate}
\end{restate}

It turns out that $2$-free one-relator groups with an exceptional intersection have presentations of a particular form that can easily be described in terms of Christoffel words. We call these primitive exceptional intersection groups; see Section \ref{pei_section} for their description. The two-generator generalised Baumslag--Solitar subgroup appearing in the statement is a $w$-subgroup in the sense of \cite{louder_21}.

Thanks to the dichotomy provided by Theorem \ref{exceptional_characterisation}, with a little work we are able to establish quasi-convexity of Magnus subgroups of all hyperbolic one-relator groups.

\begin{restate}{Theorem}{quasiconvex_magnus}
Magnus subgroups of hyperbolic one-relator groups are quasi-convex.
\end{restate}

As a consequence, we prove Theorem \ref{main_generalised}, a strengthening of the main tool from \cite{me_22}. Note that hyperbolic one-relator groups can have distorted free subgroups as there are many of hyperbolic (finitely generated free)-by-cyclic one-relator groups. See \cite{kapovich_99} for an explicit example. Theorem \ref{quasiconvex_magnus} was already known in the case of one-relator groups with torsion by Newman's Spelling Theorem \cite{newman_68}. It was also known in the case of hyperbolic one-relator groups with quasi-convex one-relator hierarchies \cite{me_22}.
 
Before stating our main result, we introduce two new families of two-generator one-relator groups. Let $p/q\in \Q_{>0}$ and denote by
\[
A_{i, j} = \{a_i, a_{i+1}, ..., a_j\}~. 
\]
The first family is the fundamental group of the following graph of groups:
\[
\begin{tikzcd}[sep=2cm]
{F(A_{0, k-1})} \arrow[rr, "{\langle A_{1, k-1}, x\rangle = \langle A_{1, k-1}, w^p\rangle}"', no head] &  & {F(A_{1, k-1})*\langle w\rangle} \arrow[rr, "{\langle A_{1, k-1}, w^q\rangle = \langle A_{1, k-1}, y\rangle}"', no head] &  & {F(A_{1, k})} \arrow[llll, "{\langle A_{0,k-1}\rangle = \langle A_{1, k}\rangle}"', no head, bend right]
\end{tikzcd}
\]
where $\{\langle x\rangle, \langle y\rangle\}$ form a malnormal family in $F(A_{0, k})$ and $x, y\notin F(A_{1, k-1})$. The upper edge homomorphism simply shifts the generators along. The second family is the fundamental group of the following graph of groups:
\[
\begin{tikzcd}[sep=2cm]
{F(A_{0, k-1})} \arrow[rr, "{\langle A_{1, k-1}, x\rangle = \langle A_{1, k-1}, x\rangle}"', no head] &  & H \arrow[rr, "{\langle A_{1, k-1}, y\rangle = \langle A_{1, k-1}, y\rangle}"', no head] &  & {F(A_{1, k})} \arrow[llll, "{\langle A_{0,k-1}\rangle = \langle A_{1, k}\rangle}"', no head, bend right]
\end{tikzcd}
\]
where $H$ takes the following form
\[
\begin{tikzcd}[sep=2cm]
{F(A_{1, k-1})} \arrow[r, "\langle z\rangle = \langle w^p\rangle", no head] & \langle w\rangle \arrow[r, "\langle w^q\rangle = \langle xy\rangle", no head] & {F(x, y)}
\end{tikzcd}
\]
and where $\langle z\rangle$ is malnormal in $F(A_{1, k-1})$ and $x, y\notin F(A_{1, k-1})$. We call these groups \emph{primitive extension groups}. 

In Section \ref{pe_section} it is shown that primitive extension groups are two-generator one-relator groups. Examples of primitive extension groups include all one-relator ascending HNN-extensions of finitely generated free groups and, in particular, $\bs(1, n)$. These correspond to the families where $x = a_0$ and $p = 1$, see Example \ref{ascending_example}. By \cite{mutanguha_21}, the hyperbolicity of one-relator groups in these subfamilies are understood. Although two-generator one-relator hyperbolic groups have received some attention \cite{kapovich_99_two, gardam_21_jsj}, there are currently no further criteria that do not rely on small cancellation-like criteria to determine when a primitive extension group is hyperbolic.

Our main result can now be stated as follows.

\begin{restate}{Theorem}{main}
A one-relator group is hyperbolic (and virtually special) if its primitive extension subgroups are hyperbolic (and virtually special).
\end{restate}

It is immediate that Gersten's conjecture \cite{gersten_92} needs only to be proved for primitive extension groups in order to hold for all one-relator groups.

\begin{restate}{Corollary}{gersten_corollary}
Gersten's conjecture is true if it is true for primitive extension groups.
\end{restate}

\subsection*{Acknowledgments} The author would like to thank Saul Schleimer for many helpful conversations, in particular for pointing him in the direction of Christoffel words. A large part of the work in this article also appears in the author's PhD thesis, completed at the University of Warwick. While carrying out part of this work, the author also received funding from the European Research Council (ERC) under the European Union's Horizon 2020 research and innovation programme (Grant agreement No. 850930).

\section{Preliminaries on free groups}

In this section, we recall some standard material from the theory of free groups and prove some technical results that will be of use to us for the proof of Theorem \ref{exceptional_characterisation}. 

If $\Sigma$ is a set, we denote by $F(\Sigma)$ the free group, freely generated by $\Sigma$. If $\Delta$ is another set, we write $F(\Sigma, \Delta) = F(\Sigma\sqcup\Delta)$.

\subsection{Proper powers}

We say an element $w\in F(\Sigma)$ is a \emph{proper power} if there is some $u\in F(\Sigma)$ and $n\geq 2$ such that $w = u^n$.

\begin{lemma}
\label{no_double_subword_1}
Let $1 \neq y^{-1}zy\in F(\Sigma)$ be freely reduced with $z$ cyclically reduced and not equal to a proper power. Then:
\begin{enumerate}
\item if $y = 1$ and there is some $i\in \Z$, $g, h\in F(\Sigma)$ such that $z^i = gzh$ is freely reduced, then $i\geq 1$ and $g,h\in \langle z\rangle$,
\item if $y\neq 1$ and there is some $i, j\in \Z$, $g, h\in F(\Sigma)$ such that $y^{-1}z^iy = gy^{-1}z^jyh$ is freely reduced, then $g, h = 1$ and $i = j$.
\end{enumerate}
\end{lemma}

\begin{proof}
Suppose for a contradiction that the first assertion does not hold. We may assume that $i = 2$ or $i = -2$. Then $z = z'z'' = z''z'$ in the first case and $z = z'z'' = (z')^{-1}(z'')^{-1}$ in the second. But then the first case cannot happen by \cite[Lemma 3]{lyndon_62} and the second case cannot happen as a non-trivial element of a free group cannot be conjugate to its own inverse. Thus, the first assertion holds.

For the proof of the second assertion, we use the first assertion and the fact that $z$ is cyclically reduced.
\end{proof}

\begin{lemma}
\label{no_double_subword_2}
Let $F(\Sigma)$ be a free group with $\Sigma = A\sqcup B$ and let $b\in \langle B\rangle$ and $z\in F(\Sigma)$ be freely reduced elements. Suppose that $z$ begins and ends with generators in $A\sqcup A^{-1}$, that $zb$ is not a proper power and that there are elements $i\in \Z$, $g, h\in F(\Sigma)$ such that $(zb)^i = gzh$ is freely reduced. Then $i\geq 1$ and $g\in \langle zb\rangle$.
\end{lemma}

\begin{proof}
By Lemma \ref{no_double_subword_1}, we may assume that $b\neq 1$. If $i\leq -1$, since $b$ does not mention any $A\sqcup A^{-1}$ generators, it follows that $z$ has a non-trivial prefix equal to its own inverse. Since this is not possible, we may also assume that $i\geq 1$.

Now suppose for a contradiction that $i\geq 1$ and $g\notin \langle zb\rangle$. Then we have equalities of the form:
\begin{align*}
z &= z_1 z_2 z_3~,\\
z &= z_3 b z_1~,
\end{align*}
for some $z_1, z_2, z_3\in F(\Sigma)$, where the words on the right hand side are freely reduced. 

We show by induction on $\abs{z_1} + \abs{z_3}$ that $z_2 = b$. If $\abs{z_1} = \abs{z_3}$, then clearly $z_2 = b$. So suppose that $\abs{z_1}<\abs{z_3}$, the other case is symmetric. Since the first letter of $z_3$ is in $A\sqcup A^{-1}$ and $b\in \langle B\rangle$, we have that $\abs{z_3}>\abs{z_1z_2}$. Now by \cite[Lemma 2]{lyndon_62}, we have $z_3 = (z_1z_2)^iz'$ for some $i\geq 1$ and where $z'$ is a proper prefix of $z_1z_2$. Then we get:
\begin{align*}
    z &= (z_1z_2)^{i+1}z'~,\\
    z &= (z_1z_2)^iz'bz_1~.
\end{align*}
By comparing suffixes, we have:
\[
z_1z_2z' = z'bz_1~.
\]
But now we obtain equalities of the same form as before, with $z'$ playing the role of $z_3$. Since $\abs{z'}<\abs{z_3}$, by induction, we see that $z_2 = b$. 

If $z_2 = b$, then $zb$ is a conjugate of itself and so must be a proper power by \cite[Lemma 3]{lyndon_62}. Thus we obtain the required contradiction and conclude that $g \in \langle zb\rangle$.
\end{proof}

\subsection{Subgroups}

A \emph{graph} is a $1$-dimensional CW-complex. A morphism of graphs $\Gamma\to \Delta$ is a map sending vertices to vertices and edges homeomorphically to edges. A morphism of graphs is an \emph{immersion}, denoted by $\immerses$, if it is locally injective. It is a fundamental observation due to Stallings \cite{sta_83} that subgroups of free groups can be represented by immersions of pointed graphs $(\Gamma, x)\immerses (\Delta, y)$. The \emph{core} of a graph $\Gamma$, denoted by $\core(\Gamma)$, is the union of the images of all immersed cycles $S^1\immerses \Gamma$. Then conjugacy classes of subgroups of free groups can be represented by immersions of core graphs $\Gamma\immerses \Delta$.

\begin{lemma}
\label{malnormal_condition_1}
Let $F(A, B, C)$ be a free group and let $x\in\langle A, B\rangle - \langle B\rangle$ and $y\in \langle B, C\rangle - \langle B\rangle$. If $H<F(A, B, C)$ is a subgroup of rank two, containing $\langle x, y\rangle$, then one of the following holds:
\begin{enumerate}
\item\label{itm:first} there are elements $u, v$ and non zero integers $i, j$, such that $H = \langle u, v\rangle$ and $x = u^i$, $y = v^j$,
\item \label{itm:second} there are elements $u, v, w$ and non zero integers $i, j$ such that $H = \langle u^{-1}wu, u^{-1}v\rangle$ and $x = u^{-1}w^iu$, $y = v^{-1}w^jv$.
\end{enumerate}
\end{lemma}

\begin{proof}
Let $\Delta$ be a rose graph with one edge for each element in $A\sqcup B\sqcup C$ so that we may identify $\pi_1(\Delta)$ with $F(A, B, C)$. Then we may represent $\langle x, y\rangle$ and $H$ by graph immersions $\Gamma\immerses\Delta$ and $\Lambda\immerses\Delta$. By assumption, $\Gamma\immerses\Delta$ factors through $\Lambda\immerses\Delta$. Thus, there must be loops based at the same vertex in $\Lambda$ with labels $x$ and $y$, covering $\Lambda$. Since a path labelled by $x$ cannot traverse any $C$-edges and a path labelled by $y$ cannot traverse any $A$-edges, it follows that there is a decomposition $Q^{(1)} = Q_1\cup Q_2$ where $\chi(Q_1), \chi(Q_2) = 0$ and $Q_1$ only contains $A$-edges and $B$-edges and $Q_2$ only contains $B$-edges and $C$-edges. Moreover, the path labelled by $x$ is supported in $Q_1$ and the path labelled by $y$ is supported in $Q_2$. If $Q_1\cap Q_2$ is connected, (\ref{itm:first}) must hold. If $Q_1\cap Q_2$ is not connected, then (\ref{itm:second}) must hold.
\end{proof}

\begin{lemma}
\label{malnormal_condition_2}
Let $F(A, B, C)$ be a free group and let $x\in\langle A, B\rangle - \langle B\rangle$, $y\in \langle B, C\rangle - \langle B\rangle$ and $z\in \langle B\rangle - 1$. If $H<F(A, B, C)$ is a subgroup of rank two, containing $\langle xy, z\rangle$, then there is an element $u$ and a non zero integer $i$, such that $H = \langle xy, u\rangle$ and $z = u^i$.
\end{lemma}

\begin{proof}
Just as in the proof of Lemma \ref{malnormal_condition_1}, let $\Delta$ be a rose graph with one edge for each element in $A\sqcup B\sqcup C$ so that we may identify $\pi_1(\Delta)$ with $F(A, B, C)$. Then we may represent $\langle xy, z\rangle$ and $H$ by graph immersions $\Gamma\immerses\Delta$ and $\Lambda\immerses\Delta$. By assumption, $\Gamma\immerses\Delta$ factors through $\Lambda\immerses\Delta$. Thus, there must be loops based at the same vertex in $\Lambda$ with labels $xy$ and $z$, covering $\Lambda$. Since $z$ only traverses $B$-edges, it follows that there is a decomposition $Q^{(1)} = Q_1\cup Q_2\cup Q_3$ where $\chi(Q_1) = 0$, $Q_1$ only contains $B$-edges, $\chi(Q_2) = \chi(Q_3) = 1$, $Q_2$ only contains $A$-edges and $B$-edges and $Q_3$ only contains $B$-edges and $C$-edges. Moreover, $z$ is supported in $Q_1$, $x$ is supported in $Q_1\cup Q_2$ and $y$ is supported in $Q_1\cup Q_3$. Now the result follows.
\end{proof}

The \emph{fibre product} $\Gamma\times_{\Delta}\Lambda$ of two graph immersions $f:\Gamma\immerses \Delta$, $g:\Lambda\immerses \Delta$ is defined as the graph with vertices:
\[
V(\Gamma\times_{\Delta}\Lambda) = \{(v_{\Gamma}, v_{\Lambda})\in V(\Gamma)\times V(\Lambda) \mid f(v_{\Gamma}) = g(v_{\Lambda})\}
\]
and edges:
\[
E(\Gamma\times_{\Delta}\Lambda) = \{(e_{\Gamma}, e_{\Lambda})\in E(\Gamma)\times E(\Lambda) \mid f(e_{\Gamma}) = g(e_{\Lambda}), \text{ respecting orientations}\}
\]
As elucidated in \cite{sta_83}, there is a correspondence between the double cosets $\pi_1(\Lambda)h\pi_1(\Gamma)$ such that $\pi_1(\Gamma)\cap \pi_1(\Lambda)^h \neq 1$ and components of the core of the fibre product graph $\core(\Gamma\times_{\Delta}\Lambda)$ given by the $\pi_1$ functor.

\section{Exceptional intersection groups}

The interactions between Magnus subgroups of one-relator groups are well understood. The following is \cite[Theorem 2]{collins_04}.

\begin{theorem}
\label{exceptional_intersection}
Let $F(\Sigma)/\normal{w}$ be a one-relator group and suppose $\Sigma = A\sqcup B\sqcup C$. If $\langle A, B\rangle$ and $\langle B, C\rangle$ are Magnus subgroups, then one of the following holds:
\begin{enumerate}
\item $\langle A, B\rangle\cap \langle B, C\rangle = \langle B\rangle$,
\item $\langle A, B\rangle\cap \langle B, C\rangle = \langle B\rangle*\Z$.
\end{enumerate}
\end{theorem}

We say the Magnus subgroups $\langle A, B\rangle$ and $\langle B, C\rangle$ have \emph{exceptional intersection} if the latter situation occurs.

\begin{definition}
A one-relator presentation $\langle \Sigma \mid w\rangle$ is an \emph{exceptional intersection presentation} if there are disjoint subsets $A, B, C\subset \Sigma$ such that $\langle A, B\rangle$ and $\langle B, C\rangle$ have exceptional intersection. A one-relator group $G$ is an \emph{exceptional intersection group} if it has an exceptional intersection presentation.
\end{definition}

The following result appears as \cite[Corollary 2.3]{collins_04}.

\begin{corollary}
\label{torsion_free_exceptional}
Exceptional intersection groups are torsion-free.
\end{corollary}

\subsection{Primitive exceptional intersection groups}
\label{pei_section}

In this section, we introduce two families of one-relator groups called primitive exceptional intersection groups. Our aim will be to show that they are precisely the exceptional intersection groups that are $2$-free. 

Before defining our groups, we first need to discuss certain primitive elements of the free group of rank two. Recall that an element of a free group $w\in F(\Sigma)$ is \emph{primitive} if $w$ forms part of a free basis for $F(\Sigma)$. Otherwise, $w$ is \emph{imprimitive}. The primitive elements of interest to us are known as Christoffel words and were first introduced in \cite{christoffel_73}. They are parametrised by a rational slope $p/q\in\Q_{>0}$. Let $\Gamma\subset \R^2$ denote the Cayley graph for $\Z^2$ on the generating set $a = (1, 0)$, $b = (0, 1)$. Let $L\subset \R^2$ be the line segment beginning at the origin and ending at the vertex $(q, p)$. Now let $P\subset \Gamma$ be the shortest length edge-path connecting the endpoints of $L$, remaining below $L$ and such that there are no integral points contained in the region enclosed by $L\cup P$. See Figure \ref{primitive_grid} for an example. The word in $a$ and $b$ traced out by $P$ is denoted by:
\[
\pr_{p/q}(a, b)~.
\]
By \cite[Theorem 1.2]{osborne_81}, every primitive element of $F(a, b)$ is conjugate into the set
\[
\left\{a^{\pm1}, b^{\pm1}, \pr_{p/q}\left(a^{\pm1}, b^{\pm1}\right) \Bigm\vert \frac{p}{q}\in \Q_{> 0}\right\}~.
\]

\begin{figure}
\centering
\includegraphics[scale=2]{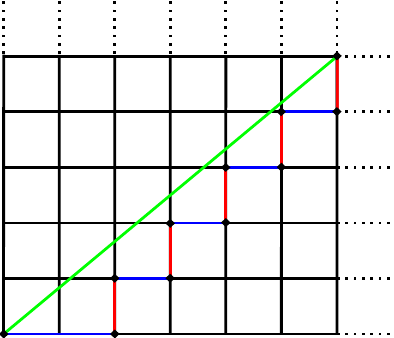}
\caption{$L$ is in green with slope $5/6$. The $a$ edges are in blue and the $b$ edges are in red, so $\pr_{5/6}(a, b) = a^2babababab$.}
\label{primitive_grid}
\end{figure}

Now consider the free group $F(A, B, C)$, freely generated by disjoint sets $A, B, C$. Let $p/q\in \Q_{>0}$ and let:
\begin{align*}
x &\in \langle A, B\rangle - \langle B\rangle~,\\
y &\in \langle B, C\rangle - \langle B\rangle~.
\end{align*}
Then we call $\pr_{p/q}(x, y)$ a \emph{primitive exceptional intersection word of the first type} if the following hold:
\begin{enumerate}
\item $\{\langle x\rangle, \langle y\rangle\}$ is a malnormal family (that is, $\langle x\rangle$ and $\langle y\rangle$ are malnormal and no conjugate of $\langle x\rangle$ intersects $\langle y\rangle$ non-trivially),
\item if $p/q = 1$, then there is no $a\in \langle A, B\rangle - \langle B\rangle$, $c\in \langle B, C\rangle - \langle B\rangle$ such that $\pr_{1}(x, y) = \pr_{1}(a, c)$ and $\{\langle a\rangle, \langle c\rangle\}$ is not a malnormal family.
\end{enumerate}
By definition, we see that $\langle x, y\rangle$ is an infinite cyclic subgroup of $G = F(A, B, C)/\normal{w}$. We find that $G$ has the following exceptional intersection:
\[
\langle A, B\rangle \cap \langle B, C\rangle =_G \langle B\rangle *\langle x^q\rangle = \langle B\rangle*\langle y^{-p}\rangle~.
\]
The following example demonstrates why we require the second condition in the definition.

\begin{example}
\label{not_first_type}
Consider the word $\pr_{1}(a^2b^{-1}, bc^2)\in F(a, b, c)$. Although the subgroups $\{\langle a^2b^{-1}\rangle, \langle bc^2\rangle\}$ form a malnormal family, we have 
\[
\pr_{1}(a^2b^{-1}, bc^2) = \pr_1(a^2, c^2) = a^2c^2~,
\]
where $\{\langle a^2\rangle, \langle c^2\rangle\}$ is not a malnormal family. Hence, $\pr_1(a^2b^{-1}, bc^2)$ is not a primitive exceptional intersection word of the first type.
\end{example}

Now let $z\in \langle B\rangle - 1$. We call $\pr_{p/q}(xy, z)$ a \emph{primitive exceptional intersection word of the second type} if the following hold:
\begin{enumerate}
\item $\langle z\rangle$ is malnormal,
\item if $p/q = k\in \N$, then there is no $a\in \langle A, B\rangle - \langle B\rangle$, $c\in \langle B, C\rangle - \langle B\rangle$ such that $\pr_{k}(xy, z) = \pr_{1}(a, c)$ and $\{\langle a\rangle, \langle c\rangle\}$ is not a malnormal family.
\end{enumerate}
By definition, we see that $\langle xy, z\rangle$ is an infinite cyclic subgroup of $G = F(A, B, C)/\normal{w}$ and so $(xy)^{-1}z(xy) =_G z$. We find that $G$ has the following exceptional intersection:
\[
\langle A, B\rangle \cap \langle B, C\rangle =_G \langle B\rangle *\langle x^{-1}zx\rangle = \langle B\rangle*\langle yzy^{-1}\rangle~.
\]

A word $w\in F(A, B, C)$ is a \emph{primitive exceptional intersection word} if $w$ is a primitive exceptional intersection word of the first or second type.

\begin{definition}
\label{pie_group}
A group $G$ is a \emph{primitive exceptional intersection group} if $G\isom F(\Sigma)/\normal{w}$ where $w$ is a primitive exceptional intersection word.
\end{definition}

\begin{example}
\label{not_second_type}
Consider the word $\pr_{k}(a^2b^2c^2, b) = a^2b^2c^2b^k\in F(a, b, c)$. Since this is equal to $\pr_{1}(a^2b^{2-k}, b^{-k}c^2b^k)$ where $\langle b^{-k}c^2b^k\rangle$ is not malnormal, it is not a primitive exceptional intersection word of the second type. However, it is also equal to the word $\pr_{2}(a^2b^{2-k}, b^{-k}cb^k)$ (or $\pr_{1/2}(a, b^2c^2b^2)$ when $k = 2$) which is a primitive exceptional intersection word of the first type. Thus
\[
G = F(a, b, c)/\normal{a^2b^2c^2b^k}
\]
is a primitive exceptional intersection group for all $k$.
\end{example}

The proof of the following theorem is rather involved and will take up the remainder of this section.

\begin{theorem}
\label{primitive_negative_immersions}
Primitive exceptional intersection groups are $2$-free.
\end{theorem}

Before proceeding with the proof, we mention some important definitions and a result from \cite{louder_21}. Define the \emph{primitivity rank} of an element $w\in F(\Sigma)$ as the following quantity:
\[
\pi(w) = \min{\{\rk(K) \mid w\in K<F, \text{ $w$ not primitive in $K$}\}}\in \N\cup \{\infty\}.
\]
Given an element $w\in F(\Sigma)$ such that $\pi(w)<\infty$, a subgroup $K<F(\Sigma)$ is a \emph{$w$-subgroup} if the following hold:
\begin{enumerate}
\item $w\in K$ and $w$ is not primitive in $K$,
\item $\rk(K) = \pi(w)$,
\item if $K<K'$, then $\rk(K)<\rk(K')$.
\end{enumerate}

In \cite{louder_21}, Louder and Wilton connect the primitivity rank $\pi(w)$ with subgroup properties of the one-relator group $F(\Sigma)/\normal{w}$. The result we shall need is the following, appearing as \cite[Theorem 1.5]{louder_21}.

\begin{theorem}
\label{lw_2_free}
A one-relator group $F(\Sigma)/\normal{w}$ is $2$-free if and only if $\pi(w)>2$. 
\end{theorem}

The idea of the proof of Theorem \ref{primitive_negative_immersions} will then be to show that if $w\in F(\Sigma)$ is a primitive exceptional intersection word, then there cannot be any two-generator $w$-subgroups.

\begin{proof}[Proof of Theorem \ref{primitive_negative_immersions}]
If $G$ is a primitive exceptional intersection group, then there is a free group $F(A, B, C)$, a rational number $p/q\in \Q_{>0}$, and elements:
\begin{align*}
x &\in \langle A, B\rangle - \langle B\rangle~,\\
y &\in \langle B, C\rangle - \langle B\rangle~,\\
z &\in \langle B\rangle - 1~,
\end{align*}
such that $G\isom F(A, B, C)/\normal{w}$ where one of the following holds:
\begin{enumerate}
\item $w = \pr_{p/q}(x, y)$ is a primitive exceptional intersection word of the first type,
\item $w = \pr_{p/q}(xy, z)$ is a primitive exceptional intersection word of the second type.
\end{enumerate}
Let us assume for sake of contradiction that $G$ is not $2$-free. Then by Corollary \ref{torsion_free_exceptional} and Theorem \ref{lw_2_free}, we have that $\pi(w) = 2$. Thus, there is some $w$-subgroup $K<F(A, B, C)$ such that $\rk(K) = 2$. We will handle the two cases at the same time as they are very similar. 

Let $\Delta$ be a rose graph with one edge for each element in $A\sqcup B\sqcup C$. The edges of $\Delta$ can then be partitioned into $A$-edges, $B$-edges and $C$-edges. Denote by $\omega:S^1\immerses \Delta$ the cycle representing $w$. There is an immersion of core graphs $\Lambda\immerses \Delta$ representing the conjugacy class of $K$ and such that there is a lift $\lambda:S^1\immerses \Lambda$ of $\omega$. Now let $\Gamma\immerses \Delta$ be the graph immersion of core graphs representing the conjugacy class of $\langle x, y\rangle<\pi_1(\Delta)$ (or $\langle xy, z\rangle$ if we are in the second case). Then there is a lift $\gamma:S^1\immerses \Gamma$ of $\omega$.

\begin{lemma}
\label{rank_1_pullback_2}
There is some connected component $\Theta\subset \core(\Gamma\times_{\Delta}\Lambda)$ such that $\Theta \isom S^1$ and $\lambda = p_{\Lambda}\mid\Theta$.
\end{lemma}

\begin{proof}
By \cite[Theorem 1]{kent_09}, either $\chi(\core(\Gamma\times_{\Delta}\Lambda)) = 0$, or $\rk(\langle \pi_1(\Lambda)^g,\pi_1(\Gamma)\rangle) = 2$ for some $g\in \pi_1(\Delta)$. In the first case, we must have that $\lambda$ factors through some component $S^1\subset \core(\Gamma\times_{\Delta}\Lambda)$. Since $G$ is torsion-free, $w$ cannot be a proper power by \cite{karrass_60}. Therefore, $\omega$ and $\lambda$ cannot factor through $S^1$ as a proper cover and we are done.

Now suppose that $\rk(\langle \pi_1(\Lambda)^g,\pi_1(\Gamma)\rangle) = 2$. By definition of $w$-subgroups, we must have that $\pi_1(\Gamma)<\pi_1(\Lambda)^g$.  Now it follows from our assumptions on $x, y, z$ and Lemmas \ref{malnormal_condition_1} and \ref{malnormal_condition_2}, that $\pi_1(\Gamma) = \pi_1(\Lambda)^g$, contradicting the fact that $w$ is not primitive in $K$.
\end{proof}

We will make use of the following factorisations of $x$ and $y$:
\begin{align*}
x &= b_1\cdot x_1^{-1}\cdot x_2\cdot x_1\cdot b_2~,\\
y &= b_3\cdot y_1^{-1}\cdot y_2\cdot y_1\cdot b_4~,
\end{align*}
as freely reduced words, where $b_1, b_2, b_3, b_4\in \langle B\rangle$, $x_1^{-1}\cdot x_2\cdot x_1$ and $y_1^{-1}\cdot y_2\cdot y_1$ do not begin or end with a $B$-letter and $x_2$ and $y_2$ are cyclically reduced. 

If $n\geq1$, denote by $u_n$ the free reduction of $x_1^{-1}x_2x_1(b_2b_1x_1^{-1}x_2x_1)^{n-1}$. Similarly, denote by $v_n$ the free reduction of $y_1^{-1}y_2y_1(b_4b_3y_1^{-1}y_2y_1)^{n-1}$.

\begin{lemma}
\label{crossing_vertices}
Let $n\geq 1$ and let $\alpha:I\immerses \Lambda$ be a path labelled by $u_n$ or $v_n$. Then $\alpha$ must traverse a vertex of degree at least three, other than at its endpoints.
\end{lemma}

\begin{proof}
We shall prove the result for $u_n$ as the other case is identical. Firstly, we show that $\Gamma$ supports at most one path with label $u_n$. If $\pi_1(\Gamma) = \langle xy, z\rangle$, then $\Gamma$ supports one path with label $u_1$ and no path with label $u_k$ for any $k\geq 2$. Let us suppose that $\pi_1(\Gamma) = \langle x, y\rangle$ and that there are two paths in $\Gamma$ with label $u_n$. Since $u_n$ begins and ends with an $A$-letter and does not contain any $C$-letters, it follows that:
\begin{enumerate}
\item if $b_2b_1\neq 1$ or $b_2b_1 = x_1 = 1$, then $u_n$ must be a subword of $u_{n+1}$ that is not a prefix or a suffix, or $u_n$ is a subword of $u_{n+1}^{-1}$,
\item if $b_2b_1 = 1$ and $x_1 \neq 1$, then $u_n$ must be a subword of $u_{m}^{\pm1}$ for some $m>n$.
\end{enumerate}
Since $\langle x\rangle$ is malnormal by definition, $x$ is not a proper power. The first situation cannot happen by Lemma \ref{no_double_subword_2}. The second situation cannot happen by Lemma \ref{no_double_subword_1}.

Now, since there is at most one path in $\Gamma$ with label $u_n$, there can be at most one lift of $\alpha$ to $\core(\Gamma\times_{\Delta}\Lambda)$ by definition of the fibre product graph. If $\alpha$ does not traverse vertices of degree three or more, except possibly at its endpoints, then each edge it traverses can have at most one preimage in $\core(\Gamma\times_{\Delta}\Lambda)$. So by Lemma \ref{rank_1_pullback_2}, since $\lambda$ is surjective, it must traverse some edge in $\Lambda$ precisely once. But then this would imply that $\lambda$ represents a primitive element of $\pi_1(\Lambda)$, contradicting the fact that $K$ was a $w$-subgroup. Therefore, we conclude that $\alpha$ must traverse a vertex of degree at least three, other than at its endpoints.
\end{proof}

We now use Lemma \ref{crossing_vertices} to derive a contradiction to the definitions of primitive exceptional intersection words. Let $\alpha:I\immerses \Lambda$ be a path satisfying the following:
\begin{enumerate}
\item $\alpha$ factors through $\lambda$,
\item $\alpha$ is labelled by $u_n$ for some $n> 0$,
\item there is no path $\alpha':I\immerses\Lambda$, strictly extending $\alpha$ and satisfying the above.
\end{enumerate}
We similarly define $\beta:I\immerses \Lambda$, replacing $u_n$ with $v_n$. Such paths exist by definition of $w$.

By Lemma \ref{crossing_vertices}, $\alpha$ and $\beta$ must traverse a vertex of degree at least three, away from their endpoints. Now the idea is to use this fact to divide $\Lambda$ according to where $\alpha$ or $\beta$ are supported. Since $\alpha$ does not traverse any $C$-edges and $\beta$ does not traverse any $A$-edges, they will block each other from traversing certain regions of $\Lambda$. 

Since $\chi(\Lambda) = -1$, and $\Lambda$ is a core graph, we only have three topologically distinct cases to consider:
\begin{enumerate}
\item $\Lambda$ is a rose graph, see Figure \ref{rose_cases},
\item $\Lambda$ is a theta graph, see Figure \ref{theta_cases},
\item $\Lambda$ is a spectacles graph, see Figure \ref{spectacles_cases}.
\end{enumerate}
Figures \ref{rose_cases}, \ref{theta_cases} and \ref{spectacles_cases} contain all the different cases, up to symmetry. Before proceeding with the case analysis, we briefly explain the diagrams. The red regions indicate sections that $\alpha$ traverses and must contain an $A$-edge; $\beta$ cannot traverse any edge in a red region. The blue regions indicate sections that $\beta$ traverses and must contain a $C$-edge; $\alpha$ cannot traverse any edge in a blue region. The yellow regions indicate sections that $\alpha$ or $\beta$ or both $\alpha$ and $\beta$ traverse. In any case, the yellow regions must only contain $B$-edges, but are allowed to have length zero when this does not change the topology of the underlying graph. The black regions indicate sections that are not traversed by either $\alpha$ or $\beta$ and are also allowed to have length zero when this does not change the topology of the underlying graph. The red vertices and blue vertices indicate the start and endpoints of $\alpha$ and $\beta$ respectively.

Topologically, in each diagram there can be at most three edges. The path $\alpha$ must leave one of these edges by Lemma \ref{crossing_vertices} and re-enter another edge, leaving enough space for $\beta$ to do the same elsewhere. Given these constraints, the reader should check that these are indeed all the cases to consider.

\begin{mycase}
\case We handle this case more in detail than the others as the arguments are mostly identical. We have three subcases to consider, according to Figure \ref{rose_cases}.

\begin{figure}
\centering
\includegraphics[scale=0.6]{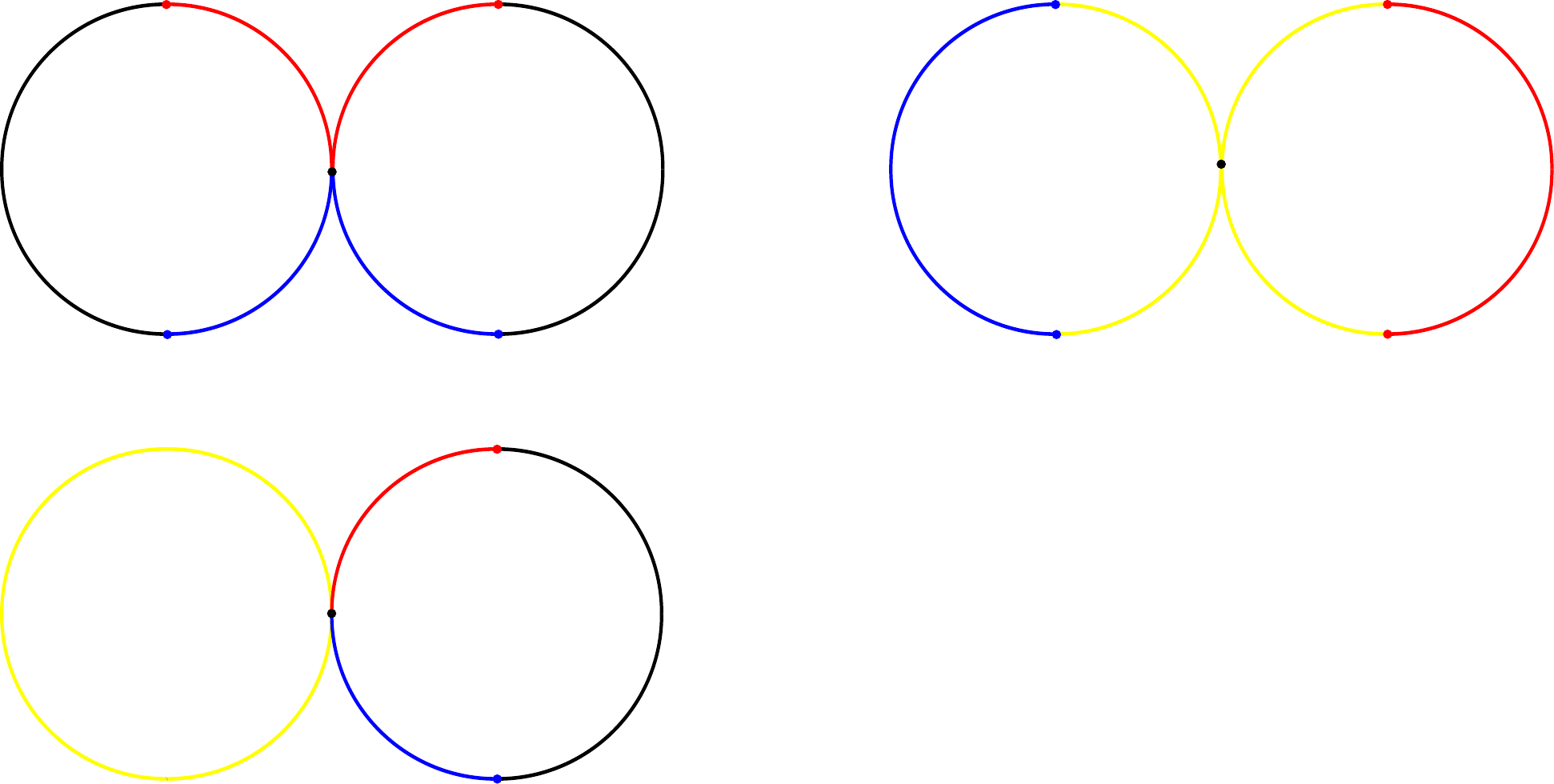}
\caption{Rose graph cases.}
\label{rose_cases}
\end{figure}

Suppose we are in the situation of the first subdiagram. When $\lambda$ traverses a red segment from a red vertex, it must then be followed by the other red segment. Otherwise we would obtain a contradiction to Lemma \ref{crossing_vertices}. Similarly for the blue segments. Thus, since $\lambda$ is primitive, it must traverse each edge precisely once. Hence, $\lambda$ would not represent a primitive element of $\pi_1(\Lambda)$ which is a contradiction.

Now consider the second subdiagram. Any (maximal) $u_n$ labelled path must begin at one red vertex and end at the other red vertex. Similarly for the $v_n$ labelled paths. But this then implies that only one power of $x$ and one power of $y$ appears in $w$. Thus, if $w$ is of the first type, then it must be equal to $xy$. If $w$ is of the second type, it must be equal to $xyz^i$ for some $i\geq 1$. By Lemma \ref{crossing_vertices}, $\lambda$ must traverse both loops at least twice. We may now deduce that there exist elements:
\begin{align*}
a &\in \langle A, B\rangle - \langle B\rangle~,\\
b &\in \langle B\rangle - 1~,\\
c &\in \langle B, C\rangle - \langle B\rangle~,
\end{align*}
such that $w = (ab^{-1})(bc)$ and that $\pi(a), \pi(c)\neq 1$. Hence, $\{\langle a\rangle, \langle c\rangle\}$ is not a malnormal family, contradicting our assumptions on $w$.

Let us move onto the third subdiagram. Similarly to the second subcase, we see that if $w$ is of the first type, it must be equal to $xy$ and if $w$ is of the second type, it must be equal to $xyz^i$ for some $i\geq 1$. From the diagram we may now deduce that there exist elements:
\begin{align*}
a &\in \langle A, B\rangle - \langle B\rangle~,\\
b &\in \langle B\rangle~,\\
c &\in \langle B, C\rangle - \langle B\rangle~,
\end{align*}
such that $w = (ab^{-1})(bc)$ and that $\langle a\rangle^g\cap \langle c\rangle\neq 1$ for some $g\in F(A, B, C)$. Hence, $\{\langle a\rangle, \langle c\rangle\}$ is not a malnormal family, contradicting our assumptions on $w$. It follows that $\Lambda$ cannot be a rose graph.

\begin{figure}
\centering
\includegraphics[scale=0.7]{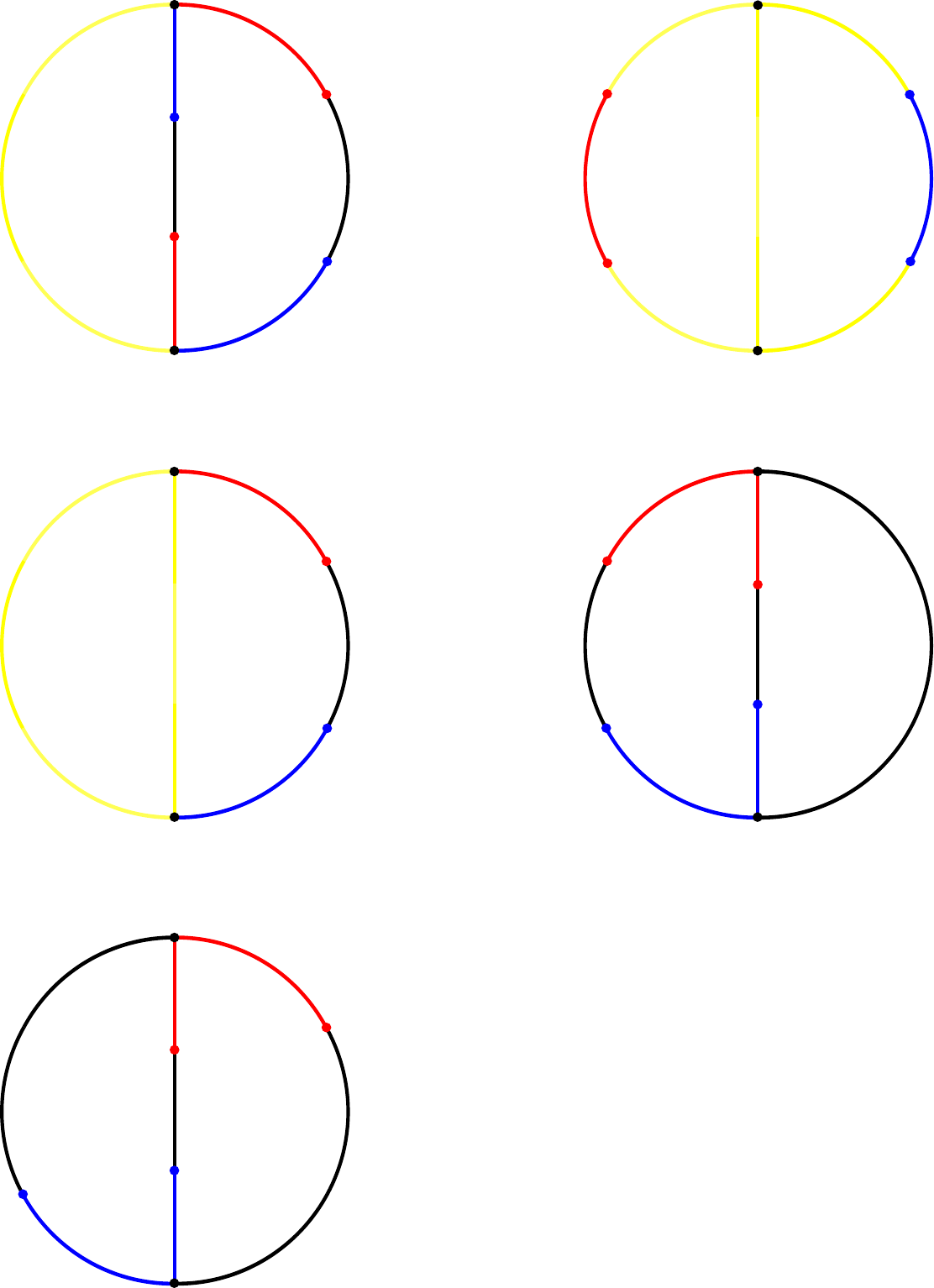}
\caption{Theta graph cases.}
\label{theta_cases}
\end{figure}

\case Consider the first subdiagram in Figure \ref{theta_cases}. By collapsing the yellow edge, we see that we may handle this case in the same way as the first subcase of the rose case. Similarly, we may reduce the second and third subcases to the second and third subcases of the rose case.

So now let us consider the new cases appearing in the fourth and fifth subdiagrams. Now if $b_2b_1\neq 1$ and $\alpha$ was labelled by $u_n$ with $n\geq 2$, then we would have that $\Lambda$ would support a path labelled by $u_1$ that does not traverse a vertex of degree at least three, away from its endpoints. But this contradicts Lemma \ref{crossing_vertices}. So if $b_2b_1 \neq 1$ and $n=1$, then there can be no other $u_m$-labelled path beginning or ending at a red vertex as $\Lambda\immerses\Delta$ is an immersion. This would imply that there can be no other $u_m$-labelled paths in $\Lambda$ for any $m\geq 1$ and so $\lambda$ traverses the red segments only once by Lemma \ref{rank_1_pullback_2}. Hence, $\lambda$ could not represent a primitive element of $\pi_1(\Lambda)$, a contradiction. So now we may assume that $b_2b_1 = 1$. By a symmetric argument, we may assume that $b_4b_3 = 1$. Similarly, we must have $x_1, y_1\neq 1$. As before, there must be at least one other path $\alpha':I\immerses\Lambda$ labelled by $u_m$ for some $m\geq 1$. We may assume that $\alpha'$ is maximal in the same sense that $\alpha$ was maximal. We see that $\alpha$ and $\alpha'$ must traverse a common segment. We now have two subcases to consider, up to symmetry, depending on whether $\alpha$ and $\alpha'$ traverse a common segment in the same direction or the opposite direction. See Figure \ref{bonus_diagram}. In either case, there can only be one lift of the segment with label $x_2^mx_1$ to $\core(\Gamma\times_{\Delta}\Lambda)$ for the following reason: the projection of any loop in $\core(\Gamma\times_{\Delta}\Lambda)$ traversing the segment labelled $x_2^mx_1$, must then traverse a blue segment labelled by $y_1^{-1}y_2^k$ for some $k\geq 0$. Since there is only one path in $\Gamma$ with this label, there can be only one lift of these segments to $\core(\Gamma\times_{\Delta}\Lambda)$. Now Lemma \ref{rank_1_pullback_2} tells us that $\lambda$ does not represent a primitive element of $\pi_1(\Lambda)$. It follows that $\Lambda$ cannot be a theta graph.

\begin{figure}
\centering
\includegraphics[scale=1.4]{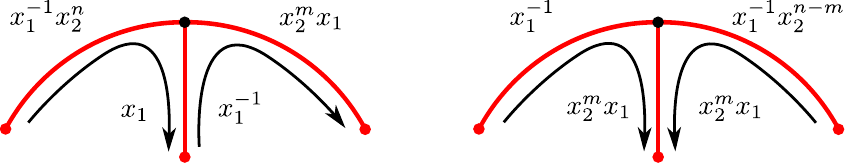}
\caption{Extra theta cases.}
\label{bonus_diagram}
\end{figure}

\begin{figure}
\centering
\includegraphics[scale=0.5]{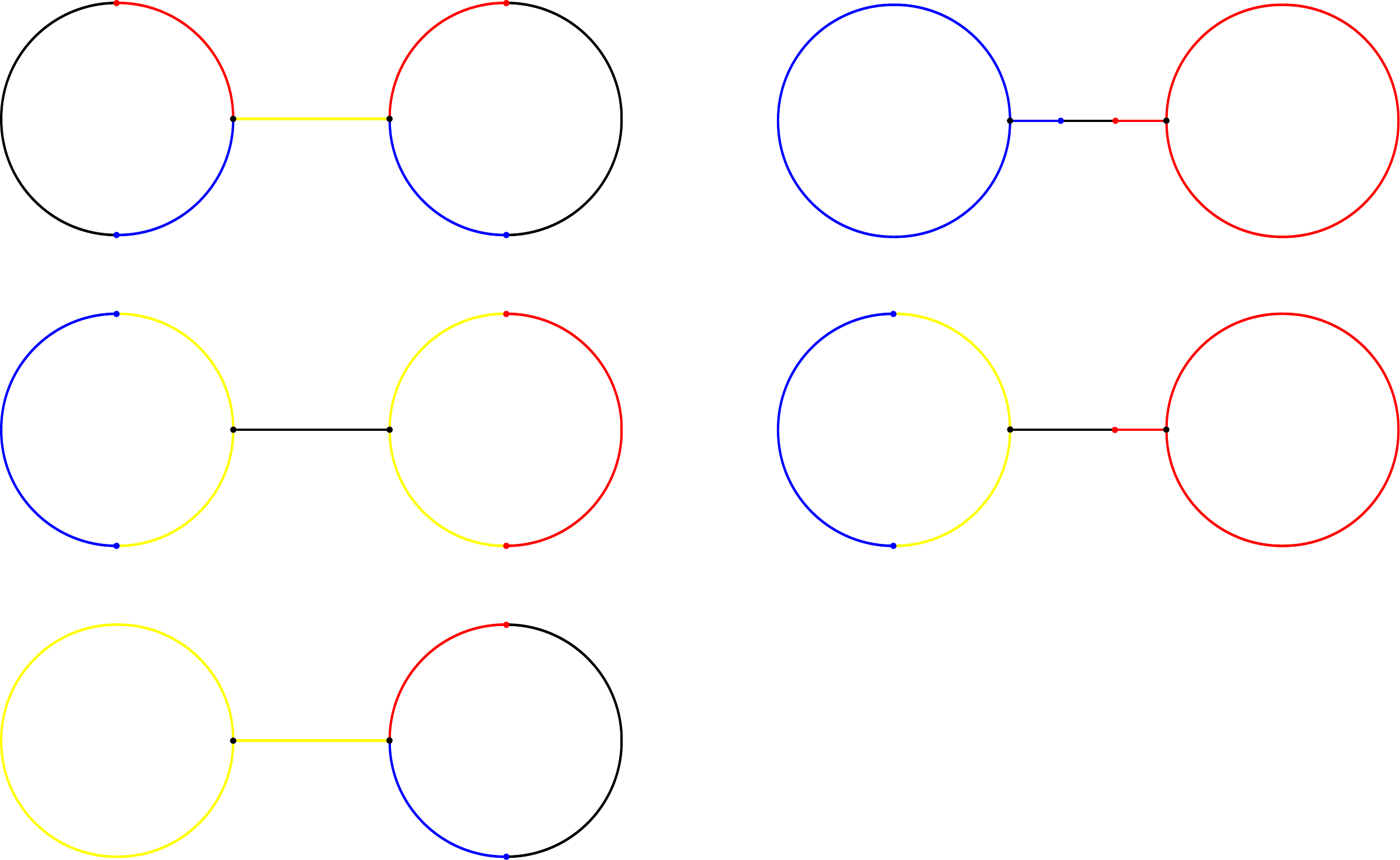}
\caption{Spectacles graph cases.}
\label{spectacles_cases}
\end{figure}

\case The first subdiagram in Figure \ref{spectacles_cases} is analogous to the first subdiagram of the rose case. The second, third and fourth subdiagrams are analogous to the second subdiagram of the rose case. The fifth subdiagram is analogous to the third subdiagram of the rose case. Hence, $\Lambda$ cannot be a spectacles graph.
\end{mycase}

Now we may conclude that there can be no $w$-subgroups of rank two and hence, $G$ must must be $2$-free.
\end{proof}

\subsection{The general case}

In this section, we consider the general case and show that primitive exceptional intersection groups are the only exceptional intersection groups which are $2$-free.

Theorem \ref{exceptional_intersection} was generalised to one-relator products in \cite{howie_05}. By specialising \cite[Theorem C]{howie_05} to the case of one-relator groups, we may obtain the following result.

\begin{theorem}
\label{exceptional_intersection_improved}
Let $F(\Sigma)/\normal{w}$ be a one-relator group and suppose $\Sigma = A\sqcup B\sqcup C$. If $\langle A, B\rangle$ and $\langle B, C\rangle$ are Magnus subgroups with exceptional intersection, then there is a monomorphism of free groups 
\[
\iota: F(a, c)\injects F(\Sigma)~,
\]
with the following properties. There is some $r\in F(a, c)$ with $\iota(r)$ conjugate to $w$ and some $m, n\neq 0$, such that one of the following holds:
\begin{enumerate}
\item $a^m = c^n$ in $F(a, c)/\normal{r}$ with $\iota(a) = x$, $\iota(c) = y$ and
\begin{align*}
x\in \langle A, B\rangle - \langle B\rangle~,\\
y\in \langle B, C\rangle - \langle B\rangle~.
\end{align*}
\item $ac^ma^{-1} = c^n$ in $F(a, c)/\normal{r}$ with $\iota(a) = xy$, $\iota(c) = z$ and
\begin{align*}
x& \in \langle A, B\rangle - \langle B\rangle~,\\
y & \in\langle B, C\rangle - \langle B\rangle~,\\
z & \in \langle B\rangle - 1~.
\end{align*}
\end{enumerate}
\end{theorem}

Using this result and the algorithm to compute the centre of a one-relator group from \cite{baumslag_67_centre}, Howie also showed that a generating set for the intersection of given Magnus subgroups is computable \cite[Theorem E]{howie_05}.

In the discussion following \cite[Theorem 4]{collins_04}, Collins points out the following.

\begin{corollary}
\label{exceptional_splitting}
Assume the notation of Theorem \ref{exceptional_intersection_improved} and suppose that we are in the first case. Denote by
\[
H = \langle B\rangle*F(a, c)/\normal{r}~.
\]
Then we have:
\[
G \isom \langle A, B\rangle \underset{\langle B, \iota(a)\rangle = \langle B, a\rangle}{*} H \underset{\langle B, c\rangle = \langle B, \iota(c)\rangle}{*} \langle B, C\rangle~.
\]
\end{corollary}

\begin{corollary}
\label{exceptional_splitting_2}
Assume the notation of Theorem \ref{exceptional_intersection_improved} and suppose that we are in the second case. Let $\iota(a) = x\cdot y$ where $x\in \langle A, B\rangle - \langle B\rangle$ and $y \in \langle B, C\rangle - \langle B\rangle$ and denote by
\[
H = \langle B\rangle\underset{\langle \iota(c)\rangle = \langle c\rangle}{*}F(a, c)/\normal{r} \underset{\langle a\rangle = \langle de\rangle}{*} F(d, e)~.
\]
Then we have:
\[
G \isom \langle A, B\rangle\underset{\langle B, x\rangle = \langle B, d\rangle}{*} H \underset{\langle B, e\rangle = \langle B, y\rangle}{*} \langle B, C\rangle~.
\]
\end{corollary}

\begin{remark}
There is a minor typographical error in the splitting provided by Collins for the second case. Corollary \ref{exceptional_splitting_2} is the corrected version.
\end{remark}

The following follows directly from Corollaries \ref{exceptional_splitting} and \ref{exceptional_splitting_2}.

\begin{corollary}
\label{induced_injective}
Assume the notation of Theorem \ref{exceptional_intersection_improved}. The monomorphism $\iota$ descends to a monomorphism of one-relator groups:
\[
\bar{\iota}:F(a, c)/\normal{r}\injects F(\Sigma)/\normal{w}~.
\]
\end{corollary}

Theorem \ref{exceptional_intersection_improved}, coupled with Corollary \ref{induced_injective}, finds us two-generator one-relator subgroups of exceptional intersection groups. We now characterise precisely what these subgroups can be.

\begin{lemma}
\label{first_exceptional_intersection}
Let $H = F(a, c)/\normal{r}$ be torsion-free such that $a^m = c^n$ in $H$ for some $m, n\neq 0$. Then one of the following hold:
\begin{enumerate}
\item $r\in F(a, c)$ is primitive and so $H\isom \Z$,
\item $H$ is non-cyclic with non-trivial centre.
\end{enumerate}
\end{lemma}

\begin{proof}
Note that $H$ has non-trivial centre as $\langle a^m\rangle$ is an infinite subgroup contained in the centre. By \cite[Chapter II Proposition 5.11]{lyn_00}, $H$ is cyclic if and only if $r$ is primitive.
\end{proof}

\begin{lemma}
\label{second_exceptional_intersection}
Let $H = F(a, c)/\normal{r}$ be torsion-free such that $a^{-1}c^ma = c^n$ in $H$ for some $m, n\neq0$. Then one of the following hold:
\begin{enumerate}
\item\label{itm:primitive} $r\in F(a, c)$ is primitive and so $H\isom \Z$,
\item\label{itm:centre} $H$ is non-cyclic with non-trivial centre,
\item\label{itm:exceptional_BS} $r$ or $r^{-1}$ is conjugate to $\pr_{p/q}(c^{-1}, a^{-1}ca)$ for some $p/q\in \Q_{>0}$ and so $H\isom BS(p, q)$.
\item\label{itm:centre_2} $r$ or $r^{-1}$ is conjugate in $F(a, c)$ to an element $r'\in \langle c, a^{-1}ca\rangle$ such that $\langle c, a^{-1}ca\rangle\isom F(c, a^{-1}ca)/\normal{r'}$ is non-cyclic with non-trivial centre.
\end{enumerate}
\end{lemma}

\begin{proof}
Suppose that $\abs{m}$ is smallest possible. If $n = m = \pm1$, then either $r$ is primitive or $H\isom \Z^2$. If $n = m\neq \pm 1$, then $H$ has non-trivial centre generated by $c^n$. So if $n = m$, we have obtained conclusion (\ref{itm:primitive}) or (\ref{itm:centre}). Now suppose that $n\neq m$. Therefore, the exponent sum of $c$ in $r$ must be non-zero. There is a single epimorphism, up to sign change, $\phi:F(a, c)\to \Z$ such that $\phi(r) = 0$, given by 
\begin{align*}
a &\to \frac{-\sigma_c(r)}{\gcd(\sigma_a(r), \sigma_c(r))}~,\\
c &\to \frac{\sigma_a(r)}{\gcd(\sigma_a(r), \sigma_c(r))}~,
\end{align*}
where $\sigma_a(r), \sigma_c(r)$ denote the exponent sums of $a$ and $c$ in $r$ respectively. But now $\phi(a^{-1}c^mac^{-n}) = 0$ from which it follows that
\[
0 = \phi(c^m) - \phi(c^n) = (m-n)\sigma_a(r)~.
\]
By Magnus rewriting, since $\sigma_a(r) = 0$, we have that 
\[
H = \langle c_0, ..., c_k, a \mid ac_0a^{-1} = c_1, ..., ac_{k-1}a^{-1} = c_k, r'\rangle,
\]
where $r'$ is the rewriting of $r$ in terms of $c = c_0, ..., c_k$. By the Freiheitssatz (see \cite[Theorem 4.10]{magnus_04}), it follows that $k = 1$. If $\langle c, aca^{-1}\rangle\isom \Z$, then we have obtained conclusion (\ref{itm:exceptional_BS}). If not, then we have obtained conclusion (\ref{itm:centre_2}).
\end{proof}

Recall that a \emph{generalised Baumslag--Solitar group} is group that splits as a graph of groups where each vertex and edge group is infinite cyclic.

\begin{proposition}
\label{gbs_relation}
Let $G = \langle a, c \mid r\rangle$ be a one-relator group in which either $a^m = c^n$ or $ac^ma^{-1} = c^n$ holds. Then $G$ is a generalised Baumslag--Solitar group.
\end{proposition}

\begin{proof}
The first case follows from Lemma \ref{first_exceptional_intersection} and \cite[Theorems 1 \& 3]{pietrowski_74}. Similarly, in the second case we may conclude that $G$ is a generalised Baumslag--Solitar group by Lemma \ref{second_exceptional_intersection}, unless $n\neq m$ and $r$ or $r^{-1}$ is conjugate in $F(a, c)$ to an element $w\in \langle c, a^{-1}ca\rangle$ such that $H = \langle c, a^{-1}ca\rangle\isom F(c, a^{-1}ca)/\normal{w}$ is non-cyclic with non-trivial centre. So let us suppose that we are in the latter case. Since $ac^ma^{-1} = c^n$ holds in $H$, the rank of the abelianisation of $H$ must be one. Then by \cite[Theorem 1]{pietrowski_74}, we have:
\[
H \isom \langle a_1, ...,. a_m \mid a_1^{p_1} = a_2^{q_1}, ..., a_{m-1}^{p_{m-1}} = a_m^{q_{m-1}}\rangle
\]
where $p_i, q_i\geq 2$ and $\gcd(p_i, q_j) = 1$ for all $i>j$. Then by \cite[Lemma 3]{pietrowski_74}, $c$ and $a^{-1}ca$ are both conjugate within $H$ to some subgroup $\langle a_i\rangle<H$. Suppose that $c^g = a_i^k$ and $(a^{-1}ca)^h = a_j^l$. Then we have that:
\[
G \isom \langle a_1, ..., a_m, b \mid a_1^{p_1} = a_2^{q_1}, ..., a_{m-1}^{p_{m-1}} = a_m^{q_{m-1}}, b^{-1}a_i^kb = a_j^l\rangle
\]
where $b = g^{-1}ah$. Thus, we may conclude that $G$ is a generalised Baumslag--Solitar group.
\end{proof}

We are now ready to prove the main result of this section.

\begin{theorem}
\label{exceptional_characterisation}
Let $G$ be an exceptional intersection group. Then one of the following holds:
\begin{enumerate}
\item\label{itm:not_neg} $G$ is a primitive exceptional intersection group and so is $2$-free,
\item\label{itm:has_GBS} there is a two-generator one-relator generalised Baumslag--Solitar subgroup $H<G$ such that every non-free two-generator subgroup of $G$ is conjugate into $H$.
\end{enumerate}
\end{theorem}

\begin{proof}
Assume that $G$ is not a primitive exceptional intersection group. Let $H<G \isom F(A, B, C)/\normal{w}$ be the two-generator subgroup from Theorem \ref{exceptional_intersection_improved}. We may assume that $H$ is maximal in the sense that there is no subgroup properly containing $H$ and that is of the same form as the two-generator subgroup from Theorem \ref{exceptional_intersection_improved}. By Proposition \ref{gbs_relation}, $H$ is a generalised Baumslag--Solitar group. Clearly $H$ has rank at most two, so we now consider the two possible cases.

Let us first assume that $H$ has rank one. Thus, $G$ has one of the following presentations:
\begin{enumerate}
\item $G \isom \langle \Sigma \mid \pr_{p/q}(x, y)\rangle$ for some $x\in \langle A, B\rangle - \langle B\rangle$ and $y\in \langle B, C\rangle - \langle B\rangle$,
\item $G\isom \langle \Sigma \mid \pr_{p/q}(xy, z)\rangle$ for some $x\in \langle A, B\rangle - \langle B\rangle$, $y\in \langle B, C\rangle - \langle B\rangle$ and $z\in \langle B\rangle$.
\end{enumerate}
where $H = \langle x, y\rangle$ in the first case, and $H = \langle xy, z\rangle$ in the second case. 

Suppose that we are in the first situation. By Definition \ref{pie_group}, we may assume that either $\{\langle x\rangle, \langle y\rangle\}$ is not a malnormal family (in $F(A, B, C)$), or that $p/q = 1$ and that there exist elements $a\in \langle A, B\rangle - \langle B\rangle$ and $c\in \langle B, C\rangle - \langle B\rangle$ such that $\pr_1(x, y) = \pr_1(a, c)$ and $\{\langle a\rangle, \langle c\rangle\}$ is not a malnormal family (in $F(A, B, C)$). As the two cases are identical, it suffices to assume that $\{\langle x\rangle, \langle y\rangle\}$ is not a malnormal family. Now, if $\{\langle x\rangle, \langle y\rangle\}$ is not a malnormal family, then either $x$ is a proper power, $y$ is a proper power, or a conjugate of $\langle x\rangle$ intersects $\langle y\rangle$ non-trivially. If either $x$ or $y$ is a proper power, by adjoining a root of $x$ or $y$ to $H$, we obtain a contradiction to maximality of $H$. If $\langle x\rangle^f\cap \langle y\rangle \neq 1$ for some $f\in F(\Sigma)$, it follows that there must be elements $g\in \langle A, B\rangle - \langle B\rangle$, $h\in \langle B, C\rangle -\langle B\rangle$, $d\in \langle B\rangle$ such that $\langle x\rangle^g\cap \langle y\rangle^{h^{-1}} < \langle d\rangle<\langle B\rangle$. Now $H$ would be properly contained in $\langle gh, h^{-1}dh\rangle$. However, since $(gh)(h^{-1}dh)^n(gh)^{-1} = (h^{-1}dh)^m$ holds for some $m, n\neq 0$, we obtain a contradiction to maximality of $H$ again.

Now suppose that we are in the second situation. As before, we may assume that $\langle z\rangle$ is not malnormal by Definition \ref{pie_group}. Then $z$ is a proper power and we contradict maximality of $H$ once again. 

We may conclude from the above that $H$ must have rank two. Since $H$ cannot be free, we have that $\pi(w) = 2$ and $G$ has a two-generator $w$-subgroup $K$ into which $H$ is conjugate by \cite[Corollary 1.10]{louder_21}. If $H\neq K$, we get a contradiction to maximality of $H$ by the definition of $w$-subgroups and Lemmas \ref{malnormal_condition_1} and \ref{malnormal_condition_2}. Thus, (\ref{itm:has_GBS}) holds by \cite[Corollary 1.10]{louder_21}.
\end{proof}

\begin{remark}
\label{w_subgroup_remark}
The $w$-subgroup from Theorem \ref{exceptional_characterisation} is always a subgroup of the form given in Theorem \ref{exceptional_intersection_improved}.
\end{remark}

It follows from \cite[Theorem 8.2]{me_22} and Theorem \ref{exceptional_characterisation} that an exceptional intersection group is hyperbolic if and only if it contains no Baumslag--Solitar subgroup. Note that this can also be derived using Corollaries \ref{exceptional_splitting} and \ref{exceptional_splitting_2} and the combination theorem \cite{bestvina_92_combination}. However, in order to prove our main results, we need the stronger dichotomy established in Theorem \ref{exceptional_characterisation}. This dichotomy is harder to establish using Corollaries \ref{exceptional_splitting} and \ref{exceptional_splitting_2} as the splittings do not satisfy the hypothesis of existing $2$-free combination theorems (see \cite{bbaumslag_68} for example).

\begin{example}
We give two examples of groups with exceptional intersection that are not $2$-free. Let $p/q\in \Q_{>0}$ and $n, m\neq 0$. Consider the group with presentation:
\[
G \isom \langle c_0, c_1 \mid \pr_{p/q}(c_0^{m}, c_1^{-n})\rangle~.
\]
The relation $c_0^{qm} = c_1^{pn}$ holds in $G$ and so it has an exceptional intersection of the first type. In \cite{meskin_73}, this group was shown to be isomorphic to a non-cyclic generalised Baumslag--Solitar group with presentation:
\[
\langle c_0, b, c_1 \mid c_0^{m} = b^{p}, b^{q} = c_1^{n}\rangle~,
\]
and so is not $2$-free.

Now consider the HNN-extension:
\begin{align*}
\langle c_0, b, c_1 , a\mid c_0^{m} = b^{p}, b^{q} = c_1^{n}, ac_0a^{-1} = c_1\rangle &\isom \langle c_0, c_1, a \mid ac_0a^{-1} = c_1, \pr_{p/q}(c_0^m, c_1^{-n}) \rangle \\
	&\isom \langle a, c \mid \pr_{p/q}(c^m, a^{-1}c^{-n}a)\rangle
\end{align*}
The relation $a^{-1}c^{mp}a = c^{nq}$ holds in this group and so
\[
\langle a_0, a_1, c \mid \pr_{p/q}(c^m, (a_0a_1)^{-1}c^{-n}(a_0a_1)\rangle
\]
has an exceptional intersection of the second type. This group also contains a non-cyclic generalised Baumslag--Solitar subgroup and so is not $2$-free.
\end{example}

\section{One-relator towers}

A \emph{one-relator complex} is a combinatorial $2$-complex with a single $2$-cell. If $X$ is a one-relator complex, we write $X = (\Gamma, \lambda)$ where $\Gamma$ is a $1$-complex and $\lambda:S^1\immerses \Gamma$ is the attaching map.

If $p:Y\immerses X$ is an infinite cyclic cover of a CW-complex, a \emph{tree domain} for $p$ is a subcomplex $Z\subset Y$ satisfying the following:
\begin{enumerate}
\item $\Z\cdot Z = Y$,
\item for all $k\in \N$ and every cell $c\subset Z$, if $k\cdot c\subset Z$, then $i\cdot c\subset Z$ for all $0\leq i\leq k$,
\item $Z\cap 1\cdot Z$ is connected and non-empty.
\end{enumerate}
By \cite[Proposition 3.10]{me_22}, if the map $\pi_1(Z\cap 1\cdot Z)\to \pi_1(Z)$ induced by inclusion is injective, then we obtain a splitting
\[
\pi_1(X)\isom \pi_1(Z)*_{\pi_1(Z\cap 1\cdot Z)}
\]
By \cite[Proposition 4.7]{me_22}, if $p:Y\immerses X$ is an infinite cyclic cover of a one-relator complex, then there always exists a one-relator tree domain $Z$ for $p$. Moreover, by the Freiheitssatz, the maps $\pi_1(Z\cap 1\cdot Z)\to\pi_1(Z)$ are always injective.

A \emph{one-relator tower} is a sequence of immersions of one-relator complexes 
\[
X_N\immerses ...\immerses X_1\immerses X_0 = X, 
\]
where $X_{i+1}\immerses X_i$ factors as 
\[
X_{i+1}\overset{\iota}{\injects} Y\overset{p}{\immerses} X_i
\]
where $p$ is an infinite cyclic covering map and $\iota$ is an inclusion of a tree domain for $p$. A one-relator tower is \emph{maximal} if $\chi(X_N) = 1$.

By the above, it follows that one-relator hierarchies correspond to iterated HNN-extensions over one-relator groups.

The following is \cite[Proposition 5.1]{me_22} and can be interpreted as a modern version of the well-known Magnus--Moldavanskii hierarchy.

\begin{proposition}
\label{maximal_tower}
Let $X$ be a finite one-relator complex. Then $X$ has a maximal one-relator tower.
\end{proposition}

If $X = (\Gamma, \lambda)$ is a one-relator complex, then for each $w$-subgroup $K<\pi_1(\Gamma)$, where $w = [\lambda]$, there is an immersion of one-relator complexes $Q = (\Omega, \omega)\immerses X$ where $\Omega\immerses \Gamma$ is the core graph immersion representing $K$. We say that $Q\immerses X$ represents a $w$-subgroup. The following is \cite[Theorem 5.15]{me_22}.

\begin{theorem}
\label{relative_hierarchy}
Let $X$ be a one-relator complex and let $Q\immerses X$ be an immersion of a one-relator complex, representing some $w$-subgroup. There exists a one-relator tower
\[
Q = X_K\immerses...\immerses X_1\immerses X_0 = X.
\]
\end{theorem}

Equipped with Theroem \ref{relative_hierarchy}, we make the following definition.

\begin{definition}
We say a one-relator tower $X_N\immerses...\immerses X_0 = X$ is \emph{factored} if for every immersion $Q\immerses X_i$ representing a $w$-subgroup of $\pi_1(X_i)$ with $\chi(Q) = 0$, either there is some $j\geq 0$ such that $X_{i+j} = Q\immerses X_i$, or $Q\immerses X_i$ factors through $X_N\immerses X_i$.
\end{definition}

The proof of the following is essentially identical to that of Proposition \ref{maximal_tower}, but we include it for completeness.

\begin{proposition}
\label{factored_tower}
Every one-relator complex has a maximal factored one-relator tower.
\end{proposition}

\begin{proof}
Let $X = (\Gamma, \lambda)$ be a one-relator complex. Denote by $\deg(\lambda)$ the largest degree covering map $S^1\immerses S^1$ that $\lambda$ factors through and denote by $X_{\lambda}$ the smallest one-relator subcomplex of $X$. Then define the quantity
\[
c(X) = \left(\frac{\abs{\lambda}}{\deg(\lambda)} - \abs{X_{\lambda}^{(0)}}, -\chi(X)\right).
\]
The proof proceeds by induction on $c(X)$. 

If $\pi_1(X)$ is $2$-free or has torsion, then any maximal one-relator tower is a factored tower by definition and so the result follows from Proposition \ref{maximal_tower}. If $Q\immerses X$ represents a $w$-subgroup, then it is clear that we must have $c(Q)\leq c(X)$ with equality if and only if $Q = X$. Thus, by Theorem \ref{relative_hierarchy}, we may assume that $\chi(X) = 0$. Then as $\chi(X)\neq 1$, $\pi_1(X)$ is indicable and there is some infinite cyclic cover $p:Y\immerses X$. By \cite[Proposition 4.9]{me_22}, there is some one-relator tree domain $Z$ for $p$ such that $c(Z)<c(X)$. Hence, by induction, the proof is complete.
\end{proof}

\subsection{Acylindrical, quasi-convex and stable one-relator towers}

Let $X_K\immerses...\immerses X_1\immerses X_0$ be a one-relator tower. Denote by $T_i$ the Bass-Serre tree associated with the splitting $\pi_1(X_i) \isom \pi_1(X_{i+1})*_{\psi_i}$. We call this an \emph{acylindrical one-relator tower} if there is some constant $k\geq 0$ such that the stabilisers of segments of length $k$ in $T_i$ are finite. We call it a \emph{quasi-convex one-relator tower} if the inclusions $\tilde{A}_{i+1}, \tilde{B}_{i+1}\injects \tilde{X_i}$ are quasi-isometric embeddings, where $A_{i+1} = X_{i+1}\cap 1\cdot X_{i+1}$ and $B_{i+1} = -1\cdot X_{i+1}\cap X_{i+1}$. This last definition is due to Wise \cite{wise_21_quasiconvex}, adapted to the one-relator case. In \cite{me_22}, a \emph{stable one-relator tower} is defined in terms of the identifying homomorphisms $\psi_i$. Since we shall not need this definition, we instead record the following, which is a reformulation of \cite[Lemma 7.4]{me_22}.

\begin{lemma}
A one-relator tower $X_N\immerses...\immerses X_1\immerses X_0$ is stable if and only if there is some $k\geq 0$ such that the pointwise stabilisers of segments of length $k$ in $T_i$ have rank at most one.
\end{lemma}

In \cite{me_22}, the author established a connection between these three types of one-relator towers. The aim of this section is to improve on that result.

\subsection{Primitive extension complexes}

A one-relator complex $X$ is a \emph{primitive extension complex} if $\chi(X) = 0$ and there is a one-relator tower $Z\immerses X$ where $\pi_1(Z)$ is $2$-free and such that
\[
\pi_1(Z\cap 1\cdot Z)\cap \pi_1(Z\cap -1\cdot Z)\neq \pi_1(-1\cdot Z\cap Z\cap 1\cdot Z),
\]
after possibly adding edges to $Z$ and extending the $\Z$-action so that $-1\cdot Z\cap Z\cap 1\cdot Z$ is connected. By Theorem \ref{exceptional_characterisation}, we can see that $\pi_1(Z)*F$ is a primitive exceptional intersection group for some finitely generated free group $F$.

Let $X = (\Gamma, \lambda)$ be a one-relator complex with $\chi(X) = 0$. Let $T\subset \Gamma$ be a spanning tree and let $\langle a, b \mid w\rangle$ be the induced one-relator presentation. If $a^m = b^n$ for some $m, n\in \Z-\{0\}$ and $\pi_1(X)$ is not cyclic, then we call $X$ a \emph{powered one-relator complex}. By Proposition \ref{gbs_relation}, $\pi_1(X)$ is a generalised Baumslag--Solitar group. Moreover, we have the following.

\begin{lemma}
\label{powered_lemma}
Powered one-relator complexes are primitive extension complexes.
\end{lemma}

\begin{proof}
Let $X = (\Gamma, \lambda)$ be a powered one-relator complex. Then there is a spanning tree in $\Gamma$ such that $a^m = b^n$ where $\langle a, b\mid w\rangle$ is the induced presentation for $\pi_1(X)$. By virtue of the fact that $a^m = b^n$, there is only a single epimorphism $\phi:\pi_1(X)\to \Z$ and by \cite{murasugi_64}, we have that $\ker(\phi)$ is finitely generated and free. Now let $p:Y\immerses X$ be the induced cyclic cover and $Z\subset Y$ a one-relator tree domain. Since $\pi_1(Y)$ is finitely generated, this in particular implies that 
\[
\pi_1(Z\cap 1\cdot Z) = \pi_1(Z) = \pi_1(-1\cdot Z\cap Z)
\]
and so that $X$ is a primitive extension complex.
\end{proof}

It turns out that primitive extension complexes obstruct stable one-relator towers.

\begin{proposition}
\label{pe_hierarchy}
Let $X$ be a one-relator complex and $X_N\immerses ...\immerses X_1\immerses X_0 = X$ a factored one-relator tower. Then one of the following holds:
\begin{enumerate}
\item $X_N\immerses ...\immerses X_1\immerses X_0 = X$ is stable,
\item $X_K\immerses ...\immerses X_1\immerses X_0 = X$ is stable and $X_K$ is a primitive extension complex for some $K\leq N$,
\item $X_K$ is a powered one-relator complex for some $K\leq N$.
\end{enumerate}
\end{proposition}

\begin{proof}
Let $p:Y\immerses X$ be the cyclic cover such that $X_1$ is a tree domain for $p$. Let $A = X_1\cap -1\cdot X_1$ and $B = X_1\cap 1\cdot X_1$, where the $\Z$ action is the covering action. Up to possibly adding finitely many edges to $X$ (and so to $X_1$), we may assume that $A\cap B$ is connected. Now by \cite[Lemma 7.8]{me_22}, either $X_1\immerses X_0 = X$ is stable, or $\pi_1(A)\cap \pi_1(B)\neq \pi_1(A\cap B)$. If $X_1\immerses X_0$ is stable, we proceed by induction. So suppose that it is not. By Theorem \ref{exceptional_intersection}, either $\pi_1(X_1)$ is $2$-free, or there is an immersion $Q\immerses X_1$ representing a $w$-subgroup such that $\chi(Q) = 0$ and $\pi_1(Q)$ is a generalised Baumslag--Solitar group. Let us first consider the latter case. By definition of factored one-relator towers, there is some $i$ such that $X_i = Q$. Moreover, by Lemmas \ref{first_exceptional_intersection} and \ref{second_exceptional_intersection}, either $X_i$ is a powered one-relator complex, or $X_{i+1}$ is. So now let us consider the case that $\pi_1(X_1)$ is $2$-free. If $\chi(X)<0$, then by definition, we have that $X$ is $2$-free. Thus, by \cite[Theorems 7.8 \& 7.9]{me_22}, we have that $X_N\immerses ...\immerses X_0 = X$ is a stable tower. If $\chi(X) = 0$, then $X$ is a primitive extension complex.
\end{proof}

\subsection{Improved one-relator towers}

Recall that a \emph{Magnus subgroup} of a one-relator group $\langle \Sigma \mid w\rangle$ is a subgroup generated by a subset of the generators $A\subset \Sigma$ such that $A$ omits a generator that appears in the cyclic reduction of $w$. Using Proposition \ref{pe_hierarchy} we may now prove that Magnus subgroups of hyperbolic one-relator groups are quasi-convex. This was previously known in the case of one-relator groups with torsion by Newman's spelling theorem and in the case of $2$-free one-relator groups by \cite[Theorems 8.1 \& 8.2]{me_22}. 

\begin{theorem}
\label{quasiconvex_magnus}
Magnus subgroups of hyperbolic one-relator groups are quasi-convex.
\end{theorem}

\begin{proof}
Let $X$ be a one-relator complex with $\pi_1(X)$ hyperbolic. By Proposition \ref{factored_tower}, there is a maximal factored one-relator tower $X_N\immerses...\immerses X_1\immerses X_0 = X$. By \cite[Corollary 7.8]{gersten_96}, $\pi_1(X_i)$ is hyperbolic for all $i\geq 1$. By Proposition \ref{pe_hierarchy}, either this tower is stable, or there is some $1\leq L<N$ such that $X_L\immerses...\immerses X_1$ is stable and $X_L$ is a primitive extension complex. 

In the first case, since hyperbolic groups cannot contain Baumslag--Solitar subgroups, the result follows by \cite[Theorems 8.1 \& 8.2]{me_22}. So now let us assume that we are in the second case. The proof proceeds by induction on tower length. For the base case $L=1$, since $\chi(X_L) = 0$, all Magnus subgroups of $\pi_1(X_L)$ must by quasi-convex as they are all cyclic. So now we assume the inductive hypothesis. In other words, that all Magnus subgroups of $\pi_1(X_1)$ are quasi-convex. Then by \cite[Theorems 3.6 \& 6.10]{me_22}, all Magnus subgroups of $\pi_1(X)$ are quasi-convex and the proof is complete.
\end{proof}

As a consequence, we may improve on the main tool developed in \cite{me_22}.

\begin{theorem}
\label{main_generalised}
Let $X$ be a one-relator complex and let $Z\immerses X$ be a one-relator tower. If $\pi_1(Z)$ is hyperbolic (and virtually special), then the following are equivalent:
\begin{enumerate}
\item\label{itm:quasiconvex} $Z\immerses X$ is a quasi-convex tower and $\pi_1(X)$ is hyperbolic (and virtually special),
\item\label{itm:acylindrical} $Z\immerses X$ is an acylindrical tower,
\item\label{itm:stable} $Z\immerses X$ is a stable tower and $\pi_1(X)$ contains no Baumslag--Solitar subgroups.
\end{enumerate}
Moreover, if any of the above is satisfied, then $\pi_1(Z)<\pi_1(X)$ is quasi-convex.
\end{theorem}

\begin{proof}
By Theorem \ref{quasiconvex_magnus}, Magnus subgroups of $\pi_1(Z)$ are quasi-convex. By \cite[Theorem 7.16]{me_22} and \cite[Theorem 13.1]{wise_21_quasiconvex}, we may establish the equivalence between (\ref{itm:acylindrical}) and (\ref{itm:stable}). Similarly, by \cite[Theorem 3.6]{me_22} and \cite[Theorem 13.1]{wise_21_quasiconvex}, we may establish the equivalence between (\ref{itm:acylindrical}) and (\ref{itm:quasiconvex}).
\end{proof}

Combining Proposition \ref{pe_hierarchy} and Theorem \ref{main_generalised} and using induction, we obtain the following corollary.

\begin{corollary}
\label{pe_corollary}
Let $X$ be a one-relator complex and suppose that $\pi_1(X)$ contains no Baumslag--Solitar subgroups.  If $X_N\immerses...\immerses X_1\immerses X_0 = X$ is a maximal factored one-relator tower, then one of the following holds:
\begin{enumerate}
\item $\pi_1(X)$ is hyperbolic and $X_N\immerses...\immerses X_1\immerses X_0 = X$ is an acylindrical tower,
\item $X_K$ is a primitive extension complex for some $K\geq 0$ and, if $\pi_1(X_K)$ is hyperbolic, then $\pi_1(X)$ is hyperbolic and $X_K\immerses...\immerses X_1\immerses X_0 = X$ is an acylindrical tower.
\end{enumerate}
\end{corollary}

\section{Primitive extension groups}
\label{pe_section}

Let $i<j\in \Z$ and denote by 
\[
A_{i, j} = \{a^{t^i}, a^{t^{i+1}}, ..., a^{t^j}\}\subset F(a, t)~. 
\]
For each $p/q\in \Q_{>0}$, we define two new families of one-relator groups. The first family is parametrised by two words 
\begin{align*}
x &\in \langle A_{0, k-1}\rangle - \langle A_{1, k-1}\rangle~,\\
y &\in \langle A_{1, k}\rangle - \langle A_{1, k-1}\rangle~,
\end{align*}
such that $\pr_{p/q}(x, y)$ is a primitive exceptional intersection word of the first type (as an element of $F(A_{0, k})$). We then define:
\begin{equation}
\label{pe_1}
E_{p/q}(x, y) = \langle a, t \mid \pr_{p/q}(x, y)\rangle~.
\end{equation}
We call this a \emph{primitive extension group of the first type}.

The second family is parametrised by three words 
\begin{align*}
x &\in \langle A_{0, k-1}\rangle - \langle A_{1, k-1}\rangle~,\\
y &\in \langle A_{1, k}\rangle - \langle A_{1, k-1}\rangle~,\\
z &\in \langle A_{1, k-1}\rangle-1~,
\end{align*}
such that $\pr_{p/q}(xy, z)$ is a primitive exceptional intersection word of the second type (as an element of $F(A_{0, k})$). We then define:
\begin{equation}
\label{pe_2}
F_{p/q}(x, y, z) = \langle a, t \mid \pr_{p/q}(xy, z)\rangle~.
\end{equation}
We call this a \emph{primitive extension group of the second type}.

\begin{definition}
A group $G$ is a \emph{primitive extension group} if it is a primitive extension group of the first or second type.
\end{definition}

It follows from Corollary \ref{exceptional_splitting} that $E_{p/q}(x, y)$ has the following graph of groups decomposition:
\[
\begin{tikzcd}[sep=2cm]
{F(A_{0, k-1})} \arrow[rr, "{\langle A_{1, k-1}, x\rangle = \langle A_{1, k-1}, w^p\rangle}"', no head] &  & {F(A_{1, k-1})*\langle w\rangle} \arrow[rr, "{\langle A_{1, k-1}, w^q\rangle = \langle A_{1, k-1}, y\rangle}"', no head] &  & {F(A_{1, k})} \arrow[llll, "{\langle A_{0,k-1}\rangle = \langle A_{1, k}\rangle}"', no head, bend right]
\end{tikzcd}
\]

Similarly, it follows from Corollary \ref{exceptional_splitting_2} that $F_{p/q}(x, y, z)$ has the following graph of groups decomposition:
\[
\begin{tikzcd}[sep=2cm]
{F(A_{0, k-1})} \arrow[rr, "{\langle A_{1, k-1}, x\rangle = \langle A_{1, k-1}, x\rangle}"', no head] &  & H \arrow[rr, "{\langle A_{1, k-1}, y\rangle = \langle A_{1, k-1}, y\rangle}"', no head] &  & {F(A_{1, k})} \arrow[llll, "{\langle A_{0,k-1}\rangle = \langle A_{1, k}\rangle}"', no head, bend right]
\end{tikzcd}
\]
where $H$ takes the following form:
\[
\begin{tikzcd}[sep=2cm]
{F(A_{1, k-1})} \arrow[r, "\langle z\rangle = \langle w^p\rangle", no head] & \langle w\rangle \arrow[r, "\langle w^q\rangle = \langle xy\rangle", no head] & {F(x, y)}
\end{tikzcd}
\]
Note that since $xy$ is primitive in $F(x, y)$, we see that 
\[
H \isom F(A_{1, k-1}, x)\underset{\langle z\rangle = \langle w^p\rangle}{\ast}\langle w\rangle.
\]

If $G$ is a finitely presented group, we denote by $\delta_G$ its Dehn function. Although the class of one-relator groups containing groups with Dehn function not bounded by any finite tower of exponentials \cite{gersten_92_l1}, the Dehn function of primitive extension groups cannot be worse than exponential.

\begin{lemma}
If $G$ is a primitive extension group, then $\delta_G(n)\preceq \exp(n)$.
\end{lemma}

\begin{proof}
Since primitive extension groups split as graphs of hyperbolic groups with finitely generated undistorted edge groups, the upper bound follows from \cite[Theorem 2]{bernasconi_94}.
\end{proof}

We say two one-relator complexes $X = (\Gamma, \lambda)$ and $Q = (\Delta, \omega)$ are \emph{Nielsen equivalent} if there is a homotopy equivalence $f:\Gamma\to \Lambda$ and a homeomorphism $s:S^1\to S^1$ such that $f\circ \lambda \simeq \omega\circ s$. This is a strong version of homotopy equivalence for one-relator complexes, introduced in \cite{louder_21}.

\begin{lemma}
\label{nielsen_equivalent}
If $X$ is a primitive extension complex, then $X$ is Nielsen equivalent to a presentation complex for (\ref{pe_1}) or (\ref{pe_2}).
\end{lemma}

\begin{proof}
Denote by $X = (\Gamma, \lambda)$. Without loss, we may assume that $\Gamma$ is a rose graph. Denote by $a$ and $b$ the two edges in $\Gamma$. This then yields a one-relator presentation $\langle a, b \mid w\rangle$ for $\pi_1(X)$.

By definition, there is a cyclic cover $p:Y\immerses X$ and a one-relator tree domain $Z\subset Y$ such that
\[
\pi_1(Z\cap 1\cdot Z)\cap \pi_1(Z\cap -1\cdot Z)\neq \pi_1(-1\cdot Z\cap Z\cap 1\cdot Z),
\]
after possibly adding edges to $Z$ and extending the $\Z$-action so that $-1\cdot Z\cap Z\cap 1\cdot Z$ is connected. 

Now the epimorphism $\pi_1(X)\to \Z$ is induced by an epimorphism $\pi_1(\Gamma)\to \Z$. Suppose that $a$ maps to zero and $b$ maps to $\pm1$ (or vice versa) under this homomorphism. Then we see that $\pi_1(Z)$ is conjugate to $\langle a, b^{-1}ab, ..., b^{-k}ab^k\rangle$ for some $k\geq 0$. Moreover, that $-1\cdot Z\cap Z\cap 1\cdot Z$ is connected and that $\langle a, b\mid w\rangle$ is thus a primitive extension presentation.

Suppose instead that $a$ maps to $p$ and $b$ maps to $-q$ under the homomorphism to $\Z$, with $p, q>0$. Let $X' = (\Gamma', \lambda')$ be the one-relator complex where $\Gamma'$ has two edges, $x$ and $y$, and $\lambda'$ is pulled back via the homotopy equivalence $f:\Gamma'\to \Gamma$ defined by mapping $x$ to $\pr_{p/q}(a, b)$ and $y$ to $\pr_{c/d}(a, b)$, where $c, d$ are integers such that $cp - dq = 1$. The two one-relator complexes are Nielsen equivalent by construction. Now the induced epimorphism $\pi_1(X')\to \Z$ maps $x$ to zero and $y$ to one and so, as before, there is a one-relator tree domain $Z'\subset Y'\immerses X'$ such that $-1\cdot Z'\cap Z'\cap 1\cdot Z'$ is connected. There is also an induced homotopy equivalence between $Z'$ and a one-relator subcomplex of $Y$. In particular, we see that we must have
\[
\pi_1(Z'\cap 1\cdot Z')\cap \pi_1(Z'\cap -1\cdot Z')\neq \pi_1(-1\cdot Z'\cap Z'\cap 1\cdot Z').
\]
Therefore, this implies that $X'$ is a presentation complex of a primitive extension group of the form (\ref{pe_1}) or (\ref{pe_2}).
\end{proof}

Before proving our main theorem, we first mention some examples of one-relator groups that are primitive extension groups.

\begin{example}
If $p, q> 0$ are coprime integers, then:
\[
E_{p/q}(a, t^{-1}a^{\pm1}t) \isom \bs(p, q^{\pm1})
\]
\end{example}

\begin{example}
A one-relator group $\langle a, b \mid w \rangle$ in which $\langle a\rangle\cap \langle b\rangle\neq 1$ is a primitive extension group by Lemma \ref{powered_lemma}. More concretely, for all non-zero integers $m, n$, the one-relator group $\langle a, b \mid a^m = b^n\rangle$ is a primitive extension group.
\end{example}

\begin{example}
\label{ascending_example}
Any one-relator group that splits as an ascending HNN-extension of a finitely generated free group (or in other words, has non-trivial BNS invariant) is a primitive extension group for the following reason: if $X$ is a one-relator complex, then by Brown's criterion \cite{brown_87}, $\pi_1(X)$ splits as an ascending HNN-extension of a free group if and only if there is a one-relator tower $Z\immerses X$ such that $\pi_1(Z) = \pi_1(Z\cap 1\cdot Z)$ or $\pi_1(Z) = \pi_1(-1\cdot Z\cap Z)$ (see also \cite[Section 5.3]{me_22}). Thus, when $\pi_1(X)$ has non-trivial BNS invariant, $\pi_1(Z)$ has an exceptional intersection and is free, so $X$ is a primitive extension complex. In fact, all such one-relator groups correspond to the subfamily $E_{p}(a, y)$.
\end{example}

\begin{theorem}
\label{main}
A one-relator group is hyperbolic (and virtually special) if its primitive extension subgroups are hyperbolic (and virtually special).
\end{theorem}

\begin{proof}
Let $G$ be a one-relator group. Suppose that all primitive extension subgroups of $G$ are hyperbolic (and virtually special). Since $\bs(1, n)$ is a primitive extension group and all Baumslag--Solitar groups contain some $\bs(1, n)$ as a subgroup, $G$ contains no Baumslag--Solitar subgroup. Now by Corollary \ref{pe_corollary} and Lemma \ref{nielsen_equivalent}, it follows that $G$ is hyperbolic (and virtually special).
\end{proof}

It is now immediate that in order to characterise hyperbolic one-relator groups, one only needs to characterise hyperbolic primitive extension groups.

\begin{corollary}
\label{gersten_corollary}
Gersten's conjecture is true if it is true for primitive extension groups.
\end{corollary}

\bibliographystyle{amsalpha}
\bibliography{bibliography}

\end{document}